\documentclass{article}
\usepackage{graphicx} 
\usepackage{amsmath,ulem}
\usepackage{mathrsfs} 
\usepackage{xcolor}
\usepackage{enumitem}
\usepackage{comment}
\usepackage{hyperref}
\usepackage{tikz-cd}
\usepackage{amssymb}
\usepackage{textcomp}
\usepackage{hyperref}
\usepackage[english]{babel}
\usepackage{amsthm}
\DeclareMathOperator{\inv}{inv}
\theoremstyle{definition}
\newtheorem{definition}{Definition}[section]
\newtheorem*{remark}{Remark}
\newtheorem*{notation}{Notation}
\newtheorem{theorem}{Theorem}[section]
\newtheorem{lemma}[theorem]{Lemma}

\newtheorem{corollary}{Corollary}[theorem]

\newtheorem{claim}[theorem]{Claim}
\newtheorem{fact}[theorem]{Fact}

\newcommand{\ignore}[1]{}
\title{Enclosing a Compact Set in an O-minimal Expansion of $(\mathbb{R},+,\cdot,0,1 <)$}
\author{Yayi Fu}
\date{}
\begin{document}
\maketitle
\begin{abstract}
Fix an o-minimal expansion $\mathcal{R}=(\mathbb{R},+,\cdot,0,1 <,...)$ of the real ordered field.
    Given $C^1$ functions $f_1,...,f_k$, 
    $g_1,...,g_k:M\to\mathbb{R}$ on a definable cell $M$,
    let $h_{i,0}$ denote $f_i$ and  $h_{i,1}$ denote $g_i$.
   Suppose that  for all $\tau\in 2^{[k]}$,
      $H_{\tau}=(h_{1,\tau(1)},..,h_{k,\tau(k)})
      :M\to \mathbb{R}^k$ is regular and proper on $M$,
      and that for all $i\in[k]$,  
      $\{f_i=0\}$ and $ \{g_i=0\}$ are connected, and
      $\{f_i=0\}\cap \{g_i=0\}=\emptyset$.
   We show that then there exists
   a sequence $(\square_{i,\epsilon}:i\in[k],\epsilon \in \{0,1\})\in\{\leq ,\geq \}^{[k]\times\{0,1\}}$ 
   such that    the enclosed region $\underset{i\in[k]}{\bigcap}
   \{ f_i\square_{i,0}
   0\} 
   \cap  
   \{ g_i\square_{i,1}
   0\}
   $ 
   is compact.
\end{abstract}

\section{Introduction}\label{secintro}
\indent

In geometry, some level sets like $\{(x,y)\in\mathbb{R}^2:x^2+y^2=1\}$ obviously enclose a compact region.

But what about level sets that are not bounded, and what if we want to  enclose a region by multiple level sets?
What can we say about the enclosed region and when is it compact?

The easiest example to think of in $\mathbb{R}^k$ is the coordinate planes $\{x_1=c_1\}, \{x_1=d_1\},...,\{x_k=c_k\},
\{x_k=d_k\}$, with $c_1,d_1...,c_k,d_k\in\mathbb{R}$ and $c_1\leq d_1,..., c_k\leq d_k$ .
For more general level sets, an immediate problem is that
the level sets may not even intersect.
For example, in $\mathbb{R}^2$, the $x$-axis $\{y=0\}$ and the level set $\{xy=1\}$ do not intersect and for the level sets   $\{y=0\}$, $\{y=1\}$, 
 $\{x=0\}$, $\{xy=1\}$,   
there's no enclosed region.

If the level sets intersect, how to control the intersection?
On one hand, to simplify the problem, we want the intersection of the level sets to be manifolds.
This can be done by assuming that the functions defining the level sets are regular.
On the other hand, we want to exclude bad examples like $\{y=0\}\cap\{y=\sin (x)\}$  in $\mathbb{R}^2$, 
where even the boundaries defined explicitly  by intersecting the level sets have infinitely many components and are difficult to study.
The finiteness for the components of the intersection of the level sets holds for any sets definable in an o-minimal expansion of the real ordered field $(\mathbb{R},+,\cdot,0,1 <)$.
Hence, we may expect that the problem would be simpler in an o-minimal setting.

Also, since it is already proved that many geometric properties are tame in  the o-minimal context, 
e.g. cell-decomposition, finite components,
curve selection, triangulation...
(see \cite{van1998tame}),
we would like to study the problem in the framework of o-minimal structure,
in the hope that the tameness of  an o-minimal structure can simplify the problem, 
and hence can  give some first-step intuition about the problem.

On the other hand, since we have seen a lot of o-minimal theorems of the type 
that in the o-minimal setting, 
stronger geometric properties hold,
it's natural to ask the reverse direction:
Can we find some geometric properties in the tame o-minimal context first,
and then see whether such properties or some weaker versions hold in a more general context?
The motivation is that since  o-minimality is often  associated with nice properties,
we expect that we can have a rich pool of properties that 
are more obvious  in the o-minimal setting, but less intuitive in general,
so we would like to use o-minimality as an intermediate step that may allow us to find more properties.

In this paper, we show the following geometric property that holds in any o-minimal expansion $\mathcal{R}=(\mathbb{R},+,\cdot,0,1 <,...)$:
\begin{theorem}
    \label{s1mainthm}
    
     Let $M$ be a $k$-cell in $\mathbb{R}^n$.
      Let $f_1,...,f_k,g_1,...,g_k$ be $C^1$ real-valued functions on $M$  definable in $\mathcal{R}$.
       Let $h_{i,0}$ denote $f_i$ and  $h_{i,1}$ denote $g_i$.
      Suppose that
      
      \begin{enumerate}
          \item [(a)]
    for all $\tau\in 2^{[k]}$,
      $H_{\tau}=(h_{1,\tau(1)},..,h_{k,\tau(k)})
      :M\to \mathbb{R}^k$ is regular on $M$;
     
      \item[(b)] 
    
     for all $\tau\in 2^{[k]}$,
      $H_{\tau}=(h_{1,\tau(1)},..,h_{k,\tau(k)})
      :M\to \mathbb{R}^k$ is proper;
      \item [(c)] for all $i\in[k]$, 
      $\{f_i=0\}$ and $ \{g_i=0\}$ are connected;
   \item [(d)]
   for all $i\in[k]$, 
      $\{f_i=0\}\cap \{g_i=0\}=\emptyset$.
      \end{enumerate}
 
      Then there exists a sequence $(\square_{i,\epsilon}:i\in[k],\epsilon \in \{0,1\})\in\{\leq ,\geq \}^{[k]\times\{0,1\}}$ 
   such that  for all $J\subseteq [k]$ and all $\tau\in2^{[k]}$,
   the enclosed region 
\[\bigcap_{i\in [k]\setminus J}
\{ f_i\square_{i,0}
   0\}
   \cap 
   \{ g_i\square_{i,1}
   0\}\cap 
\bigcap_{l\in J}\{h_{l,\tau(l)}=0\}
  \]
  is compact.

  In particular, if $J=\emptyset$
  then $\underset{i\in[k]}{\bigcap}\{ f_i\square_{i,0}
   0\}$ 
   $\cap $
   $\{ g_i\square_{i,1}
   0\}$ is compact.

\end{theorem}

     Notice that this is different from saying that each $h_{i,\tau(i)}$ is proper.
      For example, think of $f_1,...,f_k,g_1,...,g_k$ where $f_i$ is the $i$-th coordinate projection and $g_i$ is $f_i+1$.
      The hypothesis of the theorem is satisfied,
      but none of the $f_i$'s and the $g_i$'s is proper.

Section \ref{secintro} is introduction and notations.
     
Section \ref{stransint} gives background about model theory and geometry that we will use.

Section \ref{spf} is the proof for Theorem \ref{s1mainthm}.

\begin{notation}\mbox{}
    \begin{itemize}
    \item For $k\in\mathbb{N}$, $[k]$ denotes the  $\{1,...,k\}$.
        \item Let $D\subseteq \mathbb{R}^k$.
        $B_{\square \delta}^D(x)$ denotes $\{y\in D: \|x-y\|\square  \delta\}$,
        where $\square \in\{<,>,\leq ,\geq,=\}$.
        \item Given functions $f:D\to \mathbb{R}$ and $g:D\to \mathbb{R}$,
        $\{f\square g\}$ denotes 
        the set 
        $\{x\in D: f\square g\}$,
        where $\square \in\{<,>,\leq ,\geq,=\}$.
        \item Given $\square \in\{\leq ,\geq\}$, $\square'$ denotes the corresponding symbol without equality. e.g. if $\square$ is $\leq $, then $\square'$ is $<$.
        \item  $\inv:\{<,>,\leq ,\geq\}\to \{<,>,\leq ,\geq\}$ is the function that reverses the direction of the symbol. 
        e.g. $\inv(\leq )$ is $\geq $.
      
        \item  Given a manifold $M$ and $x\in M$,
        $T_x(M)$ denotes the tangent space of $M$ at $x$.
        \item Given $C\subseteq\mathbb{R}^n$ and  a function $f:C\to\mathbb{R}$,
        $\underset{\|x\|\to\infty}{\lim}
    f(x)=+\infty$  means that
    for all $r\in\mathbb{R}$ there is $N\in\mathbb{R}$ such that for all $x\in C$ with $\|x\|\geq N$, $f(x)>r$.
     $\underset{\|x\|\to\infty}{\lim}
    f(x)=-\infty$ is defined similarly.
    \end{itemize}
\end{notation}
\section{Preliminaries}
\label{stransint}
\subsection{O-minimal Background}
\indent

We give a brief list of background knowledge we will use about o-minimal structures. 
For more basic knowledge about o-minimal theory, see e.g. \cite{van1998tame} and \cite{coste1999introduction}.
\begin{definition}
  An \textit{o-minimal expansion}  of 
  the   ordered field 
  $(\mathbb{R},+,\cdot,0,1 <)$ is an 
  expansion $\mathcal{R}=(\mathbb{R},+,\cdot,0,1 <,...)$ 
  where each one-variable definable 
  (with parameters) set is a finite union of points and intervals.
\end{definition}
\begin{fact}
   \cite[Chapter~3~(2.11) CELL DECOMPOSITION THEOREM.]{van1998tame}
    
\begin{enumerate}
    \item [(I)] 
    Given any definable sets $A_1,...,A_k\subseteq \mathbb{R}^k$, there is a decomposition of $\mathbb{R}^k$ partitioning each of $A_1,...,A_k$.
\item [(II)] For each definable function $f: A \to \mathbb{R}$, $A \subseteq  \mathbb{R}^k$, there is a decomposition
$\mathcal{D}$ of $\mathbb{R}^k$ partitioning $A$ such that the restriction $f|_B: B \to \mathbb{R}$ to each cell $B\in\mathcal{D}$ with $B\subseteq A$ is continuous.
\end{enumerate}
\end{fact}
\begin{fact}
   In an o-minimal expansion $\mathcal{R}$ of the real ordered field $\mathbb{R}$,
   every definable  set 
   $A\subseteq \mathbb{R}^k$
   has finite connected components.
\end{fact}

 \begin{lemma}\label{locconn}
     Every $D\subseteq\mathbb{R}^n$ definable in $\mathcal{R}$ is locally connected.
 \end{lemma}
 
 \begin{proof}
Use cell decomposition.     
 \end{proof}
 \begin{lemma}\label{unbdconn}
    Given an unbounded definable cell $D$, there is $N\in\mathbb{N}$ such that for all $r\geq N$, $B_{\geq r}(0)\cap D$ is connected.
\end{lemma}
\begin{proof}
    Induction on  $\dim D$.
\end{proof}
\subsection{Regularity and Properness}
\indent

We list geometric facts we will use.
For the lemmas, proofs were given because we didn't find references that can be cited directly.

\begin{fact}\label{extsing}
   Let  $M$ be a manifold and let $f:M\to\mathbb{R}$ be a $C^1$ function.
   Suppose that $p\in M$ is a local maximum (or a local minimum) of $f$.
   Then $p$ is a critical point of $f$.

    In particular, if $M$ is a compact manifold, then $f$ has at least one critical point.
    \footnote{This fact is often used implicitly,
    e.g. in \cite[Part~I, Theorem~4.1]{milnor1963morse}, but the author didn't find any references that proved it explicitly.}
\end{fact}
\begin{proof}
    Let $p\in U$ and  $\varphi:\mathbb{R}^n\to U$ such that $(U,\varphi)$ is a chart.
    Then $f\circ \varphi $ is a $C^1$ function 
    on $\mathbb{R}^n$ which has a local maximum
    (or a local minimum)  at $\varphi^{-1}(p)$.
    By multi-variable Calculus,
     $\varphi^{-1}(p)$ is a critical point of $ f\circ \varphi$.
     So $p$ is a critical point of $f$.
\end{proof}

An immediate example of how to use Fact \ref{extsing} is as follows.
   \begin{lemma}\label{unbcompl}
 Let   $M$ be a manifold  which as a set
 is closed in $\mathbb{R}^n$.
         Let $f:M\to \mathbb{R}$ be a $C^1$ regular function on $M$.
         Then for all $t\in\mathbb{R}$ such that $\{f(z)<t\}\neq \emptyset$,
         $\{f(z)\leq t\}$ is unbounded.
     \end{lemma}
     \begin{proof}
     Fix $t\in\mathbb{R}$ such that $\{f(z)<t\}\neq \emptyset$.
Suppose for contradiction that    $\{f(z)\leq t\}$ is bounded.
Then by closedness of $M$ in $\mathbb{R}^n$,  $\{f(z)\leq t\}$ is compact.
Hence $f|_{ \{f(z)\leq t\}}$ reaches minimum.
Since $t\in\mathbb{R}$ satisfies that $\{f(z)<t\}\neq \emptyset$, there is $x\in \{f(z)<t\}$ such that $f(x)=\min\{f(z):z\in  M\text{ and }f(z)\leq t\}$.
(The value of $f$ is larger on $\{z\in M:f(z)=t\}$.)
But since $\{f(z)<t\}$ is open  in $M$, this $x$ would be a critical point of $f$ on $M$, contradicting the regularity assumption of $f$.
     \end{proof}

\begin{lemma}\label{exiextrem}
    Let $C\subseteq \mathbb{R}^n$ be an   unbounded set which is closed in $\mathbb{R}^n$.
    Let $f:C\to\mathbb{R}$ be a function continuous on $C$.
    
    Suppose that Im$(f)\subseteq(a,+\infty)$ (resp. Im$(f)\subseteq(-\infty,a)$) for some $a\in \mathbb{R}$,
    and $\underset{\|x\|\to\infty}{\lim}
    f(x)=+\infty$ 
    (resp.
    $\underset{\|x\|\to\infty}{\lim}
    f(x)=-\infty$).
    Then there is $z\in C$ such that $f(z)=\min$ Im$(f)$
    (resp. $z\in C$ such that $f(z)=\max$ Im$(f)$).
\end{lemma}
\begin{proof}
Let $x_0\in C$.
Let $N\in\mathbb{R}$ be large such that for all $x\in C$  with $\|x\|\geq N$, $f(x)>f(x_0)$.
Let $N'=\max\{N,\|x_0\|\}$.
Then $C\cap B_{\leq N'}(0)$ is  compact, so $f|_{C\cap B_{\leq N'}(0)}$ reaches minimum on $C\cap B_{\leq N'}(0)$.
But this is also a global minimum by the choice of $N$.

\end{proof}

\begin{remark}
    If Im$(f)$ is not bounded below (resp. bounded above),
    then $\inf$ Im$(f)$
    (resp. $\sup$ Im$(f)$) does not exist in $\mathbb{R}$.
\end{remark}

We  show that given a regular map, the restriction on a level set inherits regularity.
\begin{fact}
\label{kertan}\cite[Corollary~5.39]{lee2003smooth}

  Let $M$ be a $k$-dimensional manifold.
    Suppose that $F=(f_1,...,f_k):M\to\mathbb{R}^k$ is a regular map. 
    Then  for all $(a_1,...,a_k)\in\mathbb{R}^k$ and all $J\subseteq [k]$, $\underset{i\in  J}{\bigcap}\{f_i=a_i\}$ is a manifold and
\[
T_p( \bigcap_{i\in J}\{f_i=a_i\})=ker(d_p(\pi_{ J}\circ F)),
\]    
which has dimension $k-|J|$.
\end{fact}
\begin{lemma}\label{reglev}
    Let $M$ be a $k$-dimensional manifold.
    Suppose that $F=(f_1,...,f_k):M\to\mathbb{R}^k$ is a regular map. 
    Then for all $(a_1,...,a_{k})\in\mathbb{R}^k$ and all $I,J\subseteq [k]$ with $I\cap J=\emptyset$,
    $\pi_J\circ F$ is regular on $\underset{i\in I}{\bigcap}
    \{f_i=a_i\}$.
 
\end{lemma}
\begin{proof}

Given $p\in \underset{i\in I}{\bigcap}
\{f_i=a_i\}$, we want to show that 
$d_p(\pi_J\circ F)$ is a surjective map on $T_p( \underset{i\in I}{\bigcap}
\{f_i=a_i\})$.
By Fact \ref{kertan},
\[
T_p( \bigcap_{i\in I}\{f_i=a_i\})=ker(d_p(\pi_{I}\circ F)).
\]
Since  $d_pF:T_p(M)\to \mathbb{R}^k$ is an isomorphism and
$d_pF=(\pi_I\circ d_pF, \pi_{[k]\setminus I}\circ d_pF)$,
$\pi_{[k]\setminus I}\circ d_pF$ is an isomorphism on $ker(d_p(\pi_{I}\circ F))$.
Since $J\subseteq [k]\setminus I$, $\pi_J\circ F=\pi_J\circ\pi_{[k]\setminus I}\circ F$, and hence
$d_p(\pi_J\circ F)=\pi_J \circ \pi_{[k]\setminus I} \circ d_pF$ is surjective on  $ker(d_p(\pi_{I}\circ F))$.

\end{proof}

\section{Proof}\label{spf}

\subsection{The Boundary}
\indent

Roughly speaking,
the following lemma shows that if the  definable functions that give the boundaries are continuous,
and the set of  boundaries is non-empty, 
then for every connected component of the set partially enclosed by all but one level set,
it contains at least some boundary point.
\begin{lemma}\label{compointbdry}
     Let $D\subseteq \mathbb{R}^k$ be a connected $\mathcal{R}$-definable set  and let $f_1,...,f_k$, $g_1,...,g_k:
     D\to\mathbb{R}$ be continuous definable functions.
 
  Given $j\in[k]$, suppose that
    \[
    \underset{\{i\in[k],i\neq j\}}{\bigcup}
    \{f_i=0\}\cup\{g_i=0\}\cup
     \{f_j=0\}\neq\emptyset
     \]
  (resp.
     \(
     \underset{\{i\in[k],i\neq j\}}{\bigcup}
     \{f_i=0\}\cup\{g_i=0\}\cup
     \{g_j=0\}\neq\emptyset\)).
    
Then for all choice of  a sequence $(\square_{i,\epsilon}:i\in[k],\epsilon \in \{0,1\})\in\{\leq ,\geq \}^{[k]\times\{0,1\}}$ and for every non-empty connected component $C$ of 
\[X:=\bigcap_{\{i\in[k],i\neq j\}}
\{ f_i\square_{i,0}
   0\} 
   \cap  
   \{ g_i\square_{i,1}
   0\}
   \cap 
 \{ f_j\square_{j,0}
   0\} 
   \]
(resp. $\underset{\{i\in[k],i\neq j\}}{\bigcap}
\{ f_i\square_{i,0}
   0\} 
   \cap  
   \{ g_i\square_{i,1}
   0\}
   \cap 
 \{ g_j\square_{j,1}
   0\} $),
there is \[x\in   \underset{\{i\in[k],i\neq j\}}{\bigcup}
    \{f_i=0\}\cup\{g_i=0\}\cup
     \{f_j=0\}\]
 (resp. $x\in$
     \(
      \underset{\{i\in[k],i\neq j\}}{\bigcup}
    \{f_i=0\}\cup\{g_i=0\}\cup
     \{g_j=0\}
    \))
that is also in $C$.
\end{lemma}

\begin{proof}
Given $j\in[k]$,
WLOG, we may assume that \[ \underset{\{i\in[k],i\neq j\}}{\bigcup}\{f_i=0\}\cup\{g_i=0\}\cup
     \{f_j=0\}\neq\emptyset.\]
    Let   $(\square_{i,\epsilon}:i\in[k],\epsilon \in \{0,1\})
    \in\{\leq ,\geq \}^{[k]\times\{0,1\}}$,
    and  $C$ be a non-empty connected component of $X$.
    Suppose for contradiction that the lemma fails.
    We prove that $C$ is 
    non-empty,  closed and
    open relative to $D$.
    
    \underline{$C$ is closed in $D$}:
 Since $X$ is closed in $D$, and components of $X$ are closed in $X$,
 $C$ is closed in $D$.
 
    \underline{$C\neq\emptyset$}: 
    This is by assumption.
    
     \underline{$C$ is open in $D$}:
    Suppose that
    \[
    C\cap\left[
    \bigcup_{\{i\in[k],i\neq j\}}
   \{f_i=0\}\cup\{g_i=0\}\cup
     \{f_j=0\}
    \right]=\emptyset.
    \]
    Then since 
    $C\subseteq X$,
    \[
    C\subseteq\bigcap_{\{i\in[k],i\neq j\}}
    \{ f_i\square_{i,0}'0\}\cap
    \{g_i\square_{i,1}'0\}\cap
    \{f_j\square_{i,0}'0\}.
    \]
    
    Let $x\in C$.
    Since $\underset{\{i\in[k],i\neq j\}}{\bigcap}
    \{ f_i\square_{i,0}'0\}\cap
    \{g_i\square_{i,1}'0\}\cap
    \{f_j\square_{i,0}'0\}$
    is open in $D$, 
    there is $\delta>0$ such that $B_{<\delta}^D(x)\subseteq 
    \underset{\{i\in[k],i\neq j\}}{\bigcap}
     \{ f_i\square_{i,0}'0\}\cap
    \{g_i\square_{i,1}'0\}\cap
    \{f_j\square_{i,0}'0\}
    \subseteq X$.
   Moreover, by Lemma \ref{locconn}, we can choose $\delta>0$ such that $B_{<\delta}^D(x)$ is connected. 

    Since $C$, $B_{<\delta}^D(x)$ are connected sets,
    where 
    $x\in C\cap B_{<\delta}^D(x)\neq\emptyset$,
    $C\cup B_{<\delta}^D(x)$ is a connected subset of $X$ containing $C$.
    Since $C$ is a connected component of $X$,
    $C\cup B_{<\delta}^D(x)\subseteq C$.
    So for each $x\in C$, we found some $\delta>0$ such that 
    $B_{<\delta}^D(x)\subseteq C$. 
    It follows that $C$ is open in $D$.
    
    Since $C\neq \emptyset$ and is clopen in $D$,
    $C=D$, 
    which contradicts the assumption that 
    $\underset{\{i\in[k],i\neq j\}}{\bigcup}
    \{f_i=0\}\cup\{g_i=0\}\cup\{f_j=0\}\neq\emptyset$.
    Hence 
     \[
    C\cap\left[
    \underset{\{i\in[k],i\neq j\}}{\bigcup}
      \{f_i=0\}\cup\{g_i=0\}\cup\{f_j=0\}
      \right]
    \neq \emptyset.
    \]
\end{proof}
\begin{remark}
  $X$ is not necessarily connected, although $D$ is connected.
   On the other hand, the nice thing is that we only need connectedness and definability of $D$.
   \end{remark}
\begin{lemma}\label{kd}
    Let $k\geq 2$  and $D\subseteq\mathbb{R}^k$ be a connected unbounded  $\mathcal{R}$-definable set.

    Let   $f_1,...,f_k$, $g_1,...,g_k
    :D\to\mathbb{R}$ be continuous definable functions.
     
   Given  a sequence $(\square_{i,\epsilon}:i\in[k],\epsilon \in \{0,1\})
   \in\{\leq ,\geq \}^{[k]\times\{0,1\}}$,
   suppose that 
    \begin{enumerate}
          \item for all  $j\in[k]$,
          {$\underset{\{i\in[k],i\neq j\}}{\bigcap}$}
          $\{ f_i\square_{i,0}
   0\} 
   \cap  
   \{ g_i\square_{i,1}
   0\}
   \cap\{f_j=0\}$ is compact;
          \item for all  $j\in[k]$,
              {$\underset{\{i\in[k],i\neq j\}}{\bigcap}$}
              $\{ f_i\square_{i,0}
   0\} 
   \cap  
   \{ g_i\square_{i,1}
   0\}
   \cap\{g_j=0\}$ is compact;
          \item  for all $\delta>0$,
          \[
        B^D_{>\delta}(0)\cap  \left(
        \bigcup_{i\in [k]}\{f_i=0\}
          \cup\{g_i=0\}
          \right)\neq\emptyset;
          \]
          \item   $r>0$  satisfies that 
          $B_{\geq r}^D(0)$ is connected, and $ B_{\geq r}^D(0)$ is disjoint from
          \begin{align*}\label{bddry}
        B:=  \bigcup_{j\in[k]}
        \left[
        \bigcap_{\{i\in[k],i\neq j\}}
          \{ f_i\square_{i,0}
   0\} 
   \cap  
   \{ g_i\square_{i,1}
   0\}
          \cap\{f_j=0\}
         \cup 
          \bigcap_{\{i\in[k],i\neq j\}}
          \{ f_i\square_{i,0}
   0\} 
   \cap  
   \{ g_i\square_{i,1}
   0\}
   \cap\{g_j=0\}\right].
          \end{align*}
    \end{enumerate}
    
    Then
    \[A:=B^{D}_{\geq r}(0)\cap\bigcap_{i\in[k]}
    \{ f_i\square_{i,0}
   0\} 
   \cap  
   \{ g_i\square_{i,1}
   0\}
   =\emptyset.
    \]
   
\end{lemma}

\begin{proof}
We first show that if 1., 2. hold,
then there is at least one $r>0$ that satisfies 4, so the lemma is not trivial:

By Lemma \ref{unbdconn}, there is $\delta$ such that for all $r>\delta$,
 $B_{\geq r}^D(0)$ is connected.
By 1., 2.,
the set $B$ is compact,  so when $r$ is large, $B_{\geq r}^D(0)\cap B=\emptyset$.

 Now assume 1., 2., 3., and let $r>0$ satisfy 4..  
 We first show that $A$ does not intersect the ``boundary":
\begin{claim}\label{bdisjA}
$\left(\underset{i\in[k]}{\bigcup}\{f_i=0\}\cup\{g_i=0\}\right)\cap A=\emptyset$
\end{claim}
\begin{proof}
Suppose not. Then there is $x \in A$ such that $x$ is in, say $\{f_j=0\}$ for some $j\in[k]$.
So $x\in \underset{\{i\in[k],i\neq j\}}{\bigcap}
\{ f_i\square_{i,0}
   0\} 
   \cap  
   \{ g_i\square_{i,1}
   0\}
   \cap \{f_j=0\}$ and $\|x\|\geq r$, 
contradicting 4..

\end{proof}

  Since for each $j\in[k]$,
  \begin{align*}
   \underset{i\in [k]}{\bigcup}\{f_i=0\}
          \cup\{g_i=0\}
          &=
          \underset{\{i\in[k],i\neq j\}}{\bigcup}
  \{f_i=0\}\cup\{g_i=0\}\cup\{f_j=0\}\\
  &\cup \underset{\{i\in[k],i\neq j\}}{\bigcup}
  \{f_i=0\}\cup\{g_i=0\}\cup\{g_j=0\},
  \end{align*}
  WLOG, we may assume in 3. that for some $j\in[k]$,
  \[ B^D_{\geq r}(0)\cap \left(
  \underset{\{i\in[k],i\neq j\}}{\bigcup}
  \{f_i=0\}\cup\{g_i=0\}\cup\{f_j=0\}
 \right)
          \neq\emptyset.
          \]
          By Lemma \ref{compointbdry}, connectedness of $B^D_{\geq r}(0)$ 
          and 3., for every non-empty connected component of $X:=
          B^{D}_{\geq r}(0)\cap
  \underset{\{i\in[k],i\neq j\}}{\bigcap}
  \{ f_i\square_{i,0}
   0\} 
   \cap  
   \{ g_i\square_{i,1}
   0\}
   \cap \{f_j\square_{j,0}0\}$, there is 
     $x\in \underset{\{i\in[k],i\neq j\}}{\bigcup}
   \{f_i=0\}\cup\{g_i=0\}\cup\{f_j=0\}$ in the component.
   
    Suppose for contradiction that   $A\neq \emptyset$.
    Let $a\in A\subseteq X$.
    Let $C$ be the connected component of $X$ which contains $a$.
  By Lemma \ref{compointbdry}, connectedness of $B^D_{\geq r}(0)$ and 3., 
  \[C\cap\left(\bigcup_{\{i\in[k],i\neq j\}}
  \{f_i=0\}\cup\{g_i=0\}\cup\{f_j=0\}
  \right)\neq\emptyset.\]

    Let $z\in C\cap\left(
    \underset{\{i\in[k],i\neq j\}}{\bigcup}
     \{f_i=0\}\cup\{g_i=0\}\cup\{f_j=0\}
     \right)$.
   Then $z\in \{g_j\inv(\square_{j,1})'0\}$.
   Otherwise, \[z\in A\cap \left(\bigcup_{i\in[k]}\{f_i=0\}
\cup\{g_i=0\}\right),
\]
contradicting Claim \ref{bdisjA}.
Since $a\in A$,  $a\in \{ g_j\square_{j,1}0\}$.
Because $C$ is connected, and $a,z\in C$,
by Intermediate Value Theorem applied to the continuous function $g_j$,
there is $w\in C\cap \{g_j=0\}$.
Since \[ w\in C\cap \{g_j=0\}
\subseteq A\cap
\left(\bigcup_{i\in[k]}
\{f_i=0\}\cup\{g_i=0\}
\right),\]
this contradicts Claim \ref{bdisjA}.
So we must have $A=\emptyset$.
\end{proof}
\begin{remark}
\mbox{}
\begin{enumerate}

    \item 
    In the proof, we used the full power of hypotheses 1. and 2., since in the intersection $C\cap\left(\underset{\{i\in[k],i\neq j\}}{\bigcup}
   \{f_i=0\}\cup\{g_i=0\}\cup\{f_j=0\}
   \right)$, we didn't know from which  $\{f_i=0\}$ 
or $\{g_i=0\}$ we would get the point $z$.
\item To show that the ``enclosed" region $\underset{i\in[k]}{\bigcap}
    \{ f_i\square_{i,0}
   0\} 
   \cap  
   \{ g_i\square_{i,1}
   0\}$ is compact,
   by the lemma, we just need to check the hypotheses 1., 2., 3..
\end{enumerate}
\end{remark}

\subsection{The Enclosed Region}
\indent

In this section, we will prove the main theorem and other properties that may be of independent interest.

  Let $f_1,...,f_k,g_1,...,g_k$ be $C^1$ real-valued functions on $\mathbb{R}^k$  definable in $\mathcal{R}$.
    Let $h_{i,0}$ denote $f_i$ and  $h_{i,1}$ denote $g_i$.
We will work under the assumption 
that
\begin{align}\label{hass}
   &\text{ for all $\tau\in 2^{[k]}$,
      $H_{\tau}=(h_{1,\tau(1)},..,h_{k,\tau(k)})
      :\mathbb{R}^k\to \mathbb{R}^k$ is regular;} 
      \notag\\
     &\text{
     for all $\tau\in 2^{[k]}$,
      $H_{\tau}=(h_{1,\tau(1)},..,h_{k,\tau(k)})
      :\mathbb{R}^k\to \mathbb{R}^k$ is proper;} \notag\\
       &\text{
       for all $i\in[k]$, 
      $\{f_i=0\}$ and $ \{g_i=0\}$ are connected;
       }\notag\\
       \tag{$*$}
       &\text{
       for all $i\in[k]$, 
      $\{f_i=0\}\cap \{g_i=0\}=\emptyset$.
       }
\end{align}
We list properties (\ref{c1}), (\ref{c2}), (\ref{c3}) that we are going to study.
In particular, we will want a sequence $(\square_{i,\epsilon}:i\in[k],\epsilon \in \{0,1\})
   \in\{\leq ,\geq \}^{[k]\times\{0,1\}}$
   that satisfies property (\ref{c2}):      \begin{align}
          &\text{For all $J\subsetneq [k]$ and all
$\tau\in 2^{J}$,
every non-empty component of $\underset{j\in J}{\bigcap}\{h_{j,{\tau(j)}}=0\}$}\notag\\
&\text{is unbounded. }\label{c1}
      \end{align}
      \begin{align}
         &\text{For  all $J\subsetneq [k]$ and all
$\tau\in 2^{J}$,
given $i\in [k]\setminus J$ and $C$ a component of}\notag
\\&\text{$ \underset{j\in  J}{\bigcap}\{h_{j,{\tau(j)}}=0\}$,
         $\{f_i\square_{i,0}0\}\cap\{g_i\inv(\square_{i,1})0\}
         \cap
         C$ and
         $\{f_i\inv(\square_{i,0})0\}\cap\{g_i\square_{i,1}0\}
         \cap
        C$ }\notag
        \\
        &\text{are unbounded. }\label{c2}
       \end{align}
       
       \begin{align}
         &\text{  For  all $J\subsetneq [k]$ and all
$\tau\in 2^{J}$,
given $i\in [k]\setminus J$,
and $C$ a component of }\notag\\
&\text{$ \underset{j\in  J}{\bigcap}\{h_{j,{\tau(j)}}=0\}$,
$f_i, g_i
    $ are regular on  $C$.}
    \label{c3}
      \end{align}
\begin{remark}\mbox{}
    \begin{enumerate}
        \item By regularity of the $H_\tau$'s in the assumption (\ref{hass}), 
        property (\ref{c3}) holds automatically by Lemma \ref{reglev} 
        and o-minimality that implies finite components for a definable set.
        \item    In property (\ref{c1}), (\ref{c2}),  
        it is equivalent to just consider $J\subsetneq[k]$ with $|J|=k-1$.
        
    \end{enumerate}
\end{remark}

We first prove lemmas about the level sets.

\begin{lemma}
    \label{ubdom}
  
 Let $f_1,...,f_k,g_1,...,g_k$ be $C^1$ real-valued functions on $\mathbb{R}^k$ 
 definable in $\mathcal{R}$ such that the assumption (\ref{hass}) holds.
 Then given 
$\tau\in 2^{J}$ for some $J\subsetneq [k]$, for all sequence $(\epsilon_j:j\in J)\in \mathbb{R}^J$,
every non-empty component of $\underset{j\in J}{\bigcap}\{h_{j,{\tau(j)}}=\epsilon_j\}$ is unbounded.    
\end{lemma}
\begin{proof}
First observe that since $\mathcal{R}$ is o-minimal, 
$\underset{j\in J}{\bigcap}\{h_{j,{\tau(j)}}=\epsilon_j\}$
has finitely many components, so each of its components is clopoen relative to
$\underset{j\in J}{\bigcap}\{h_{j,{\tau(j)}}=\epsilon_j\}$. 

Suppose that a component $C$ of $\underset{j\in J}{\bigcap}\{h_{j,{\tau(j)}}=\epsilon_j\}$ is  bounded.
Then since $C$ is closed in $\underset{j\in J}{\bigcap}\{h_{j,{\tau(j)}}=\epsilon_j\}$, 
which is closed in $\mathbb{R}^k$, $C$ is compact.
So given $j\in [k]\setminus J$, $h_{j,\tau(j)}$ reaches extrema on the compact set $C$, 
which would give a singular point in $C$ by Fact \ref{extsing}.
Since $C$ is also open relative to  $\underset{j\in J}{\bigcap}\{h_{j,{\tau(j)}}=\epsilon_j\}$, 
this contradicts the regularity of $h_{j,\tau(j)}$ on
$\underset{j\in J}{\bigcap}\{h_{j,{\tau(j)}}=\epsilon_j\}$ given by Lemma \ref{reglev}.
\end{proof}

 \begin{lemma}\label{whoim}
Let $f_1,...,f_k,g_1,...,g_k$ be $C^1$ real-valued functions on $\mathbb{R}^k$ 
definable in $\mathcal{R}$ such that the assumption (\ref{hass}) holds.
 Then given 
$\tau\in 2^{J}$ with $J\subsetneq [k]$, 
for every non-empty component $C$ of $\underset{j\in J}{\bigcap}\{h_{j,{\tau(j)}}=0\}$ 
and every $h_{j,\epsilon}$ with $j\in [k]\setminus J$, $\epsilon\in\{0,1\}$,
\[
 \text{Im}(h_{j,\epsilon}|_{C} )
  =\mathbb{R}.
  \]
   \end{lemma}
   \begin{proof}
   Fix $\tau\in 2^{J}$ for some $J\subsetneq [k]$.
   By  Lemma \ref{ubdom}, given $C$  a  component of $\underset{j\in J}
   {\bigcap}\{h_{j,{\tau(j)}}=0\}$, 
   $C$ is unbounded and is closed in $\mathbb{R}^k$.
   Fix $h_{j,\epsilon}$ with $j\in [k]\setminus J$, $\epsilon\in\{0,1\}$.
  Notice that
if $[k]\setminus J\cup\{j\}\neq \emptyset$, for each $i\in [k]\setminus J\cup\{j\}$,
we can choose $\epsilon_{i}\in $ Im $h_{i,0}(C)$  
such that $C\cap \underset{\{i\in[k]:i\notin J\cup\{j\}\}}{\bigcap}\{h_{i,0}=\epsilon_i\}\neq \emptyset$:
Take $x\in C $ and for each $i\in [k]\setminus J\cup\{j\}$, 
let $\epsilon_i=h_{i,0}(x)$.
Then $x\in C\cap \underset{\{i\in[k]:i\notin J\cup\{j\}\}}{\bigcap}\{h_{i,0}=\epsilon_i\}$.
If $[k]\setminus J\cup\{j\}= \emptyset$, then we do not choose any $\epsilon_i$.
 
Observe that for each $x\in  C\cap \underset{\{i\in[k]:i\notin J\cup\{j\}\}}{\bigcap}\{h_{i,0}=\epsilon_i\}$,
if $D$ is the component  in $ \underset{j\in J}{\bigcap}
\{h_{j,\tau(j)}=0\}
\cap 
\underset{\{i\in[k]:i\notin J\cup\{j\}\}}{\bigcap}
\{h_{i,0}=\epsilon_i\}$ that contains $x$,
then $D\subseteq
C\cap \underset{\{i\in[k]:i\notin J\cup\{j\}\}}{\bigcap}\{h_{i,0}=\epsilon_i\}$.
By Lemma \ref{ubdom}, 
$D$ and thus $C\cap \underset{\{i\in[k]:i\notin J\cup\{j\}\}}{\bigcap}\{h_{i,0}=\epsilon_i\}$ are unbounded.
Hence, by properness in  assumption (\ref{hass}), \[
h_{j,\epsilon} (C\cap \underset{\{i\in[k]:i\notin J\cup\{j\}\}}{\bigcap}\{h_{i,0}=\epsilon_i\})
\]
is unbounded.
In particular, $h_{j,\epsilon}(C)$ is unbounded.
  If $h_{j,\epsilon}(C)$ is not all of $\mathbb{R}$,
  then by connectedness of $C$, $h_{j,\epsilon}(C)\subseteq (c,+\infty)$ for some $c\in\mathbb{R}$, 
  or  $\subseteq (-\infty,c)$ for some $c\in\mathbb{R}$.
  WLOG, we may assume that  $h_{j,\epsilon}(C)\subseteq (c,+\infty)$ for some $c\in\mathbb{R}$.
  Since $h_{j,\epsilon}(C)\subseteq (c,+\infty)$ and by properness, 
  $\underset{\|x\|\to\infty}{\lim}
    |h_{j,\epsilon}(x)|=+\infty$ in $C\cap \underset{\{i\in[k]:i\notin J\cup\{j\}\}}{\bigcap}\{h_{i,0}=\epsilon_i\}$,
    we have  $\underset{\|x\|\to\infty}{\lim}
    h_{j,\epsilon}(x)=+\infty$ in $C\cap \underset{\{i\in[k]:i\notin J\cup\{j\}\}}{\bigcap}\{h_{i,0}=\epsilon_i\}$.
  So $h_{j,\epsilon}$  has a local extremum in $C\cap \underset{\{i\in[k]:i\notin J\cup\{j\}\}}{\bigcap}\{h_{i,0}=\epsilon_i\}$  by Lemma \ref{exiextrem},
  and hence a  singular point in $C\cap \underset{\{i\in[k]:i\notin J\cup\{j\}\}}{\bigcap}\{h_{i,0}=\epsilon_i\}$.
  But by Lemma \ref{reglev}, $h_{j,\epsilon}$ is regular
  on $\underset{j\in J}{\bigcap}\{h_{j,{\tau(j)}}=0\}
  \cap \underset{\{i\in[k]:i\notin J\cup\{j\}\}}{\bigcap}\{h_{i,0}=\epsilon_i\}$,
  and thus on $C\cap \underset{\{i\in[k]:i\notin J\cup\{j\}\}}{\bigcap}\{h_{i,0}=\epsilon_i\}$ which is open relative to
  $\underset{j\in J}{\bigcap}\{h_{j,{\tau(j)}}=0\}
  \cap \underset{\{i\in[k]:i\notin J\cup\{j\}\}}{\bigcap}\{h_{i,0}=\epsilon_i\}$, a contradiction.

   \end{proof}
   
   Lemma \ref{ubdom} and Lemma \ref{whoim} imply the following corollary:
   \begin{corollary}\label{leset}
      Let $f_1,...,f_k,g_1,...,g_k$ be $C^1$ real-valued functions on $\mathbb{R}^k$ 
definable in $\mathcal{R}$ such that the assumption (\ref{hass}) holds.

 Then given 
$\tau\in 2^{J}$ for some $J\subsetneq [k]$, 
$\underset{j\in J}{\bigcap}\{h_{j,{\tau(j)}}=0\}\neq \emptyset$ and every component of it is unbounded.
   \end{corollary}
\begin{proof}
    For the non-emptyness, use an inductive argument on $|J|$ and apply Lemma \ref{whoim}.
    The unboundedness follows from Lemma \ref{ubdom}.
 \end{proof}

  So far, we have been proving properties for level sets $\underset{j\in J}{\bigcap}\{h_{j,{\tau(j)}}=0\} $ with $J\subsetneq [k]$.
  We now deal with the case when $J=[k]$.
   \begin{corollary}
\label{uniqint}   Let $f_1,...,f_k,g_1,...,g_k$ be $C^1$ real-valued functions on 
$\mathbb{R}^k$  definable in $\mathcal{R}$ such that the assumption (\ref{hass}) holds.
   Then for all $\tau\in[k]$ and  all $i\in[k]$, 
   if $C$ is a connected component of $\underset{\{l\in[k]:l\neq i\}}{\bigcap}
   \{h_{l,\tau(l)}=0\}$,
   then the sets
   \[
   \{f_i= 0\}\cap 
C\text{ and }
\{  g_i=0\}\cap 
C
\]   are singletons.    \end{corollary}
\begin{proof}

Fix a $\tau\in[k]$ and  an $i\in[k]$.
   Let  $C$ be a connected component of 
   $\underset{\{l\in[k]:l\neq i\}}{\bigcap}\{h_{l,\tau(l)}=0\}$, 
   and let $\gamma:\mathbb{R}\to C$ be a homeomorphic parametrization of  $C$.
   By regularity and o-minimality, $  \{f_i= 0\}\cap 
C$ is a finite set of discrete points.
By Lemma \ref{whoim}, $  \{f_i= 0\}\cap 
C$ has at least one point, say $p=\gamma(a)$ for some $a\in \mathbb{R}$.
Suppose that $q=\gamma(b)$ is another point in $  \{f_i= 0\}\cap 
C$  such that for all $x$ between $a$ and $b$, $f_i(\gamma(x))\neq 0$.
WLOG, we may assume that $a<b$.
Then $\gamma([a,b])$ is compact and $f_i$ reaches extrema on it.
If an extremum is in the interior, then $f_i$ has a singular point on $\underset{\{l\in[k]:l\neq i\}}{\bigcap}
\{h_{l,\tau(l)}=0\}$, 
contradicting Lemma \ref{reglev}.
So both the maximum and the minimum are realized by $p,q$.
But then, $f_i(y)=0$ for all $y\in\gamma([a,b])$, a contradiction.
Hence there is at most one point in  $  \{f_i= 0\}\cap 
C$.

The proof for $ \{g_i= 0\}\cap 
C$ is similar.
\end{proof}

   The following theorem says that the assumption (\ref{hass}) implies 
that property (\ref{c2}) holds for some a sequence 
$(\square_{i,\epsilon}:i\in[k],\epsilon \in \{0,1\})\in\{\leq ,\geq \}^{[k]\times\{0,1\}}$.
\begin{theorem}
    \label{exim}
   
      Let $f_1,...,f_k,g_1,...,g_k$ be $C^1$ real-valued functions 
      on $\mathbb{R}^k$  definable in $\mathcal{R}$ such that the assumption (\ref{hass}) holds.

   Then there exists
a sequence $(\square_{i,\epsilon}:i\in[k],\epsilon \in \{0,1\})\in\{\leq ,\geq \}^{[k]\times\{0,1\}}$ such that property (\ref{c2}) holds.
\end{theorem}

\begin{proof}
It suffices to show that there exists a sequence $(\square_{i,\epsilon}:i\in[k],\epsilon \in \{0,1\})\in\{\leq ,\geq \}^{[k]\times\{0,1\}}$ such that the case for $|J|=k-1$ holds.

Let $\tau_0:[k]\to\{0\}$ be constantly zero.
For each $j\in[k]$,  fix $C_j$ a component of $ \underset{\{i\in[k]:i\neq j\}}{\bigcap}\{h_{i,{\tau_0(i)}}=0\}$.   
To choose the $\square$'s, for each $j\in[k]$, 
consider the singletons $\{p_j\}=\{f_j=0\}\cap C_j$ and  $\{q_j\}=\{g_j=0\}\cap C_j$.
Take $\square_{j,0}\in\{\leq,\geq \}$ such that $f_j(q_j)\square_{j,0}0$;
take $\square_{j,1}\in\{\leq,\geq \}$ such that $g_j(p_j)\square_{j,1}0$.

It remains to check that such choice of the $\square_{j,\epsilon}$'s work for all $\tau\in2^{[k]}$.
We do this by showing the following claim.

\begin{claim}
    Given $\tau\in2^{[k]}$, $j\in[k]$, and $D$ a component of 
    $ \underset{\{i\in[k]:i\neq j\}}{\bigcap}\{h_{i,{\tau(i)}}=0\}$,
by Corollary \ref{uniqint},
let $\{p\}=\{f_j=0\}\cap D$ and  $\{q\}=\{g_j=0\}\cap D$.
    Then 
    \begin{enumerate}
        \item 
    $f_j(q)\square_{j,0}0$ and $g_j(p)\square_{j,1}0$;
    \item 
    $\{f_j\square_{j,0} 0\}\cap \{g_j\inv(\square_{j,1}) 0\}\cap
         D$ and
        $\{f_j\inv(\square_{j,0}) 0\}\cap \{g_j\square_{j,1}0\}\cap
         D$
       are unbounded.
    \end{enumerate}
\end{claim}
\begin{proof}
    [Proof of Claim]
      Suppose for contradiction that,
        say $f_j(q)\inv(\square_{j,0})0$.
       Since $q,q_j\in\{g_j=0\}$,
       by the connectedness of $\{g_j=0\}$ given in the assumption (\ref{hass}),
  there is $x\in \{g_j=0\}$ such that $f_j(x)=0$.
       Hence $\{f_j=0\}\cap \{g_j=0\}
       \neq\emptyset$, contradicting  the assumption (\ref{hass}).
Similarly, $g_j(p)\square_{j,1}0$.

For 2., let $\gamma:\mathbb{R}\to D$ be a continuous parametrization of $D$.
Let $p=\gamma(a)$ and $q=\gamma(b)$ for some  $a,b\in\mathbb{R}$.
WLOG, we may assume that $a<b$.
Given $x\in (b,+\infty)$, if  $f_j(\gamma(x))\inv(\square_{j,0})0$,
then there is $y\in (b,x)$ such that 
$f_j(\gamma(y))=0$, contradicting Corollary \ref{uniqint}.
So $\gamma((b,+\infty))\subseteq 
\{f_j\square_{j,0} 0\}\cap
         D$.
         Similarly, since $g_j(p)\square_{j,1}0$,
         for all $x\in(-\infty,b)$, $g_j(\gamma(x))\square_{j,1}0$.
         By Lemma \ref{whoim}, there is $y\in (b,+\infty)$ such that 
$g_j(\gamma(y))\inv(\square_{j,1})0$.
By IVT and Corollary \ref{uniqint}, 
        $\gamma((b,+\infty))\subseteq 
\{g_j\inv(\square_{j,1}) 0\}\cap
         D$.
         
Hence, $\gamma((b,+\infty))\subseteq 
\{f_j\square_{j,0} 0\}\cap 
\{g_j\inv(\square_{j,1}) 0\}\cap
         D$, which is thus unbounded.
         The same argument shows that  $\gamma((-\infty,a))\subseteq 
         \{f_j\inv(\square_{j,0}) 0\}\cap \{g_j(\square_{j,1}) 0\}\cap
         D$,
         and hence it
       is unbounded.
\end{proof}

\end{proof}

 The following lemma is the base case of Theorem \ref{finstep}.
   We write it as a separate lemma for clarity.
  \begin{lemma}\label{ubtocm1d}

       Let $f_1,...,f_k,g_1,...,g_k$ be $C^1$ real-valued functions on $\mathbb{R}^k$  definable in $\mathcal{R}$ such that the assumption (\ref{hass}) holds, 
and let
   a sequence $(\square_{i,\epsilon}:i\in[k],\epsilon \in \{0,1\})\in\{\leq ,\geq \}^{[k]\times\{0,1\}}$ satisfy  property (\ref{c2}).
   
   Then for   all $\tau\in 2^{[k]}$ and all $i\in[k]$,
      \[K:=
\{f_i\square_{i,0}0\}\cap\{g_i\square_{i,1}0\}
\cap 
\underset{\{l\in[k]:l\neq i\}}{\bigcap}
\{h_{l,\tau(l)}=0\}
  \] is compact.
  \footnote{$\bigcap_{l\neq i}\{h_{l,\tau(l)}=0\}$ is a $1$-manifold.}
\end{lemma}
\begin{proof}

Fix $i\in[k]$ and  $\tau\in 2^{[k]}$.
By o-minimality, it suffices to show that  each connected component $C$ 
of $\bigcap_{l\neq i}\{h_{l,\tau(l)}=0\}$ satisfies that 
$\{f_i\square_{i,0}0\}\cap\{g_i\square_{i,1}0\}\cap 
C $ is compact.
   
   Fix $C$  an unbounded component of $\bigcap_{l\neq i}\{h_{l,\tau(l)}=0\}$.
   Since connected components  of $\underset{\{l\in[k]:l\neq i\}}{\bigcap}
   \{h_{l,\tau(l)}=0\}$ 
   are closed relative to $\underset{\{l\in[k]:l\neq i\}}{\bigcap}
   \{h_{l,\tau(l)}=0\}$,
   $C$ is a closed subset in $\mathbb{R}^{k}$. 
   To show that $\{f_i\square_{i,0}0\}\cap
   \{g_i\square_{i,1}0\}\cap 
C $ is compact,
by continuity of the $f_i$'s and $g_i$'s, it suffices to show that it is bounded.

  Let $C$ be parametrized by a continuous map $\gamma:\mathbb{R}\to C$.
   By Corollary \ref{uniqint},  
   let $p=\gamma(a)$ be the unique point in 
      $\{  f_i=0 \}\cap 
C$;
$q=\gamma(b)$ be the unique point in $\{ g_i= 0\}\cap 
C$.
   By  the assumption (\ref{hass}),
   $p\neq q$.

     By property (\ref{c2}),   $\{f_i\square_{i,0}0\}\cap\{g_i\inv(\square_{i,1})0\}
         \cap
         C$ and
         $\{f_i\inv(\square_{i,0})0\}\cap\{g_i\square_{i,1}0\}
         \cap
        C$ 
       are unbounded.
 \begin{itemize}
     \item \underline{Case $a< b$}:
    Suppose for contrdiction that $g_i(p)\inv(\square_{i,1}) 0$     and $f_i(q)\square_{i,0} 0$.
 \begin{claim}
     For all $x\in (b,+\infty)$, $g_i(\gamma(x))\square_{i,1}' 0$.
 \end{claim}
 \begin{proof}
     [Proof of Claim]
We first prove that for all $x\in (-\infty,b)$, $g_i(\gamma(x))\inv(\square_{i,1}) 0$.
  Otherwise, there is  $y\in (-\infty,b)$ such that $g_i(\gamma(y))\square_{i,1} 0$, and by IVT, there is $z\in (-\infty,b)$, $g_i(\gamma(z))= 0$, a contradiction.

Now by Lemma \ref{whoim},
since Im($g_i|_C$)$=\mathbb{R}$,
there is $y\in (b,+\infty)$ such that $g_i(\gamma(y))\square_{i,1}' 0$.
By the same argument using IVT, 
for all $x\in (b,+\infty)$, $g_i(\gamma(x))\square_{i,1}' 0$.
 \end{proof}
  Similarly, for all $x\in (-\infty,a)$, $f_i(\gamma(x))\inv(\square_{i,0}) 0$.
    Hence,
    \(\{f_i\square_{i,0}0\}
    \cap
    \{g_i\inv(\square_{i,1})0\}
         \cap
         C
         \subseteq \gamma([a,b]),\) contradicting property (\ref{c2}).
         
           Similarly, if $g_i(p)\square_{i,1} 0$ and 
           $f_i(q)\inv(\square_{i,0}) 0$,
           then   $\{f_i\inv(\square_{i,0})0\}
           \cap\{g_i\square_{i,1}0\}
         \cap
        C
        \subseteq \gamma([a,b])$, contradicting property (\ref{c2}).
          
          Hence, we must have  $g_i(p)\square_{i,1} 0$ and $f_i(q)\square_{i,0} 0$, or $g_i(p)\inv(\square_{i,1}) 0$ and $f_i(q)\inv(\square_{i,0}) 0$.
          By the same reasoning using IVT, this implies that $\{f_j\square_{j,0}0\}\cap\{g_j\square_{j,1}0\}\cap 
C \subseteq\gamma([a,b])$,
or 
 $\{f_j\square_{j,0}0\}\cap\{g_j\square_{j,1}0\}\cap 
C =\emptyset$, and hence the set is bounded.
     \item \underline{Case  $a> b$}:
     Similar.
 \end{itemize}
      Hence, $\{f_i\square_{i,0}0\}\cap\{g_i\square_{i,1}0\}\cap 
C $ must be  bounded.
      
\end{proof}

 Now we prove the main theorem.
 The idea is that when we have compact building blocks of lower dimension, 
 we can use these  building blocks to enclose a compact region.

\begin{theorem}\label{finstep}
  Let $M$ be a $k$-cell in $\mathbb{R}^n$.
     Let $f_1,...,f_k,g_1,...,g_k$ be $C^1$ real-valued functions on $M$  definable in $\mathcal{R}$ such that the assumption (\ref{hass}) holds.

      Then there exists a sequence $(\square_{i,\epsilon}:i\in[k],\epsilon \in \{0,1\})\in\{\leq ,\geq \}^{[k]\times\{0,1\}}$ 
      such that for all $J\subseteq [k]$ and all $\tau\in2^{[k]}$,
      the enclosed region \[\bigcap_{i\in [k]\setminus J}
\{f_i\square_{i,0}0\}\cap\{g_i\square_{i,1}0\}\cap 
\bigcap_{l\in J}\{h_{l,\tau(l)}=0\}
  \] is compact.

  In particular, if $J=\emptyset$
  then \[
  \bigcap_{i\in [k]}
\{f_i\square_{i,0}0\}\cap\{g_i\square_{i,1}0\}
  \] is compact.
\end{theorem}
\begin{proof} 
We first show the case where $f_1,...,f_k,g_1,...,g_k$ are defined on $\mathbb{R}^k$.

By Theorem \ref{exim}, we can fix a sequence $(\square_{i,\epsilon}:i\in[k],\epsilon \in \{0,1\})\in\{\leq ,\geq \}^{[k]\times\{0,1\}}$ 
      such that property (\ref{c2}) holds. 
      We will prove by induction on $t=k-|J|$.
      
      \underline{$t=0$}: 
      This is Corollary \ref{uniqint}.
      
      \underline{$t=1$}: 
$J=[k]\setminus\{i\}$ for some $i\in[k]$.
Done in Lemma \ref{ubtocm1d}.

\underline{$t+1$}: 
Fix $\tau\in 2^{[k]}$ and $J\subseteq [k]$ with cardinality $k-t$.
To show that 
\[\underset{i\in[k]\setminus J}{\bigcap}
\{f_i\square_{i,0}0\}\cap\{g_i\square_{i,1}0\}\cap 
\bigcap_{l\in J}\{h_{l,\tau(l)}=0\}
\]
is compact, we want to show that for each component $C$ of $\underset{l\in J}{\bigcap}\{h_{l,\tau(l)}=0\}$, \\
$\underset{i\in[k]\setminus J}{\bigcap}
\{f_i\square_{i,0}0\}\cap\{g_i\square_{i,1}0\}\cap 
C$ is compact.
By Lemma \ref{kd},
it suffices to check that
\begin{enumerate}
          \item for all $j\in[k]\setminus J$,
         $\underset{\{i\in[k]\setminus J:i\neq j\}}{\bigcap}$
                  $\{f_i\square_{i,0}0\}\cap\{g_i\square_{i,1}0\}
          \cap\{f_j=0\}
          \cap C$ is compact;
          \item for all  $j\in [k]\setminus J$,
           $\underset{\{i\in[k]\setminus J:i\neq j\}}{\bigcap}$
                  $\{f_i\square_{i,0}0\}\cap\{g_i\square_{i,1}0\}\cap\{g_j=0\}\cap C$ is compact;
          \item for all $\delta>0$,
          \[
         B^{C}_{>\delta}(0)\cap \left(\bigcup_{i\in[k]\setminus J}
        \{ f_i=0\}
    \cup
    \{ g_i=0\}\right)
   \neq\emptyset.
          \]
    \end{enumerate}
Condition 3. is given by Corollary \ref{leset}. 
So we just check 1. and 2..

\begin{enumerate}
          \item Given $j\in[k]\setminus J$,
          \[ 
          C\cap
        \bigcap_{i\in[k]\setminus (J\cup\{j\})}
          \{f_i\square_{i,0}0\}\cap\{g_i\square_{i,1}0\}
          \cap\{f_j=0\}
          \] 
          is a closed subset  of 
          \[
          \bigcap_{i\in[k]\setminus (J\cup\{j\})}
            \{f_i\square_{i,0}0\}\cap\{g_i\square_{i,1}0\}
          \cap\bigcap_{l\in J\cup\{j\}}
          \{h_{l,\tau'(l)}=0\},
          \]
          where $\tau'$ is a function in $2^{[k]}$ such that $\tau|_{J}=\tau'|_{J}$ and 
          $\tau'(j)=0$.
          This is is compact by the inductive hypothesis.
          \item Given $j\in[k]\setminus J$,
          \[
          C\cap
 \bigcap_{i\in[k]\setminus (J\cup\{j\})}
          \{f_i\square_{i,0}0\}\cap\{g_i\square_{i,1}0\}
              \cap\{ g_j=0\}\]
              is a closed subset  of 
              \[
          \bigcap_{i\in[k]\setminus (J\cup\{j\})}
            \{f_i\square_{i,0}0\}\cap\{g_i\square_{i,1}0\}
          \cap\bigcap_{l\in J\cup\{j\}}
      \{h_{l,\tau'(l)}=0\}, 
          \]
          where $\tau'$ is an function in $2^{[k]}$ such that $\tau|_{J}=\tau'|_{J}$ and 
          $\tau'(j)=1$.
          This is is compact by inductive hypothesis.
       
    \end{enumerate}
    Hence, the inductive step holds
    and we conclude that $\underset{i\in[k]\setminus J}{\bigcap}
\{f_i\square_{i,0}0\}\cap\{g_i\square_{i,1}0\}\cap 
C$ is compact.
Since this is true for all connected components of
$\underset{l\in J}{\bigcap}\{h_{l,\tau(l)}=0\}$ 
and $\underset{l\in J}{\bigcap}\{h_{l,\tau(l)}=0\}$ has finitely many components, 
    $\underset{i\in[k]\setminus J}{\bigcap}
    \{f_i\square_{i,0}0\}\cap\{g_i\square_{i,1}0\}
    \cap 
\underset{l\in J}{\bigcap}\{h_{l,\tau(l)}=0\}$ 
is compact.
We completed the inductive step.

In the general case,
let $M$ be a $k$-cell in $\mathbb{R}^n$.
      Let $f_1,...,f_k,g_1,...,g_k$ be $C^1$ functions on $M$  definable in $\mathcal{R}$.    
     Take $h:\mathbb{R}^k\to M$  a diffeomorphism.
      Apply the theorem to the case for $\mathbb{R}^k$.
      We can choose  a sequence $(\square_{i,\epsilon}:i\in[k],\epsilon \in \{0,1\})\in\{\leq ,\geq \}^{[k]\times\{0,1\}}$
      satisfying that  
      \[
     A:= \bigcap_{1\leq i\leq k}\{z\in\mathbb{R}^k:  
     f_i\circ h(z)\square_{i,0}0
     \text{ and }
     g_i\circ h(z)\square_{i,1}0\}
      \] is compact.
      Then $h(A)$  is equal to 
      \[
    \bigcap_{1\leq i\leq k}\{x\in M:  f_i(x)\square_{i,0}0
     \text{ and }
     g_i(x)\square_{i,1}0\}.
      \]
Compactness of $h(A)$ follows immediately from continuity of $h$ and compactness of $A$.
 
\end{proof}

\newpage
\bibliographystyle{alpha}
\bibliography{ref}

\end{document}